\documentclass[12pt]{amsart}
\usepackage{amssymb,amsthm}
\usepackage{amsmath}

\newtheorem{thm}{\bf Theorem}
\newtheorem{lem}{\bf Lemma}
\newtheorem{cor}{\bf Corollary}
\newtheorem{pro}{\bf Proposition}

\theoremstyle{remark}

\theoremstyle{definition}

\begin{document}
\title{Convergence of weighted averages of relaxed projections}
\author{Ryszard Szwarc}
\date{}

\address{Institute of Mathematics, University of Wroc\l aw, pl.\ Grunwaldzki 2/4, 50-384 Wroc\l aw, Poland
\newline   and \newline \indent Institute of Mathematics and Computer Science, University of Opole, ul. Oleska 48,
45-052 Opole, Poland}
 \email{szwarc2@gmail.com}
\thanks{Supported by European Commission Marie Curie Host
Fellowship for the Transfer of Knowledge ``Harmonic Analysis, Nonlinear Analysis and
Probability''  MTKD-CT-2004-013389
 and by
MNiSW Grant N201 054 32/4285.}

\keywords{convex set, weighted projections, relaxation} \subjclass[2000]{Primary
49M45}
\begin{abstract}
The convergence of the algorithm for solving convex feasibility problem is studied
by the method of   sequential averaged and relaxed projections. Some results of
H. H. Bauschke and J. M. Borwein are generalized by introducing new methods.
Examples illustrating these generalizations are given.
\end{abstract}
\maketitle
\section{Introduction and preliminaries}
Many problems in applied mathematics deal with finding a point in the intersection of
a family of convex sets in Euclidean or Hilbert space. The solution can be achieved in algorithmic way
  as a limit of composition
of projection onto these convex sets. Due to its importance the problem has been studied heavily
for many years. We refer to \cite{BB} where the reader can find an extensive account of
theorems and literature related to the problem, as well as a general   approach which captures the
earlier methods and results. In the present work we will generalize, in case of finite dimension,
one of the main theorems
of \cite{BB}, concerning the convergence of algorithm.
 Firstly we will improve the estimate   for the average of
relaxed  projections, secondly we will admit repetitive control, and lastly we
will make use of  perturbation theorem,  which  allows to ignore projections with
small weight coefficients. The first generalization  is valid also for infinite dimensional inner product spaces.
All  these was possible by applying new techniques
and new proofs as well.

   For a closed convex set $C\subset
 \mathbb{R}^d$ let $P_C$ denote the projection onto $C.$ For $x\in   \mathbb{R}^d$ the symbol
$d(x,C)$ will denote the distance from $x$ to $C.$   Assume we are dealing
with a fixed finite family of closed convex sets $C_1,\ C_2,\ldots, C_N.$ For a sequence
of relaxation parameters $\alpha_1,\ldots,\alpha_N$ such that $0\le \alpha_i\le 2$ and
numbers $\lambda_1,\ldots,\lambda_N$ such that $0\le \lambda_i\le 1,$ $\sum_{i=1}^N\lambda_i=1$ we will
consider the weighted averages
$$ \sum_{j=1}^N\lambda_j\{(1-\alpha_j)I+\alpha_jP_{C_j}\}=
 I-\sum_{j=1}^N \lambda_j\alpha_j (I-P_{C_j}).$$ It is well known that every operator $P_{C_j}$ is firmly nonexpansive
 (see \cite[Facts 1.5]{BB}), thus
 these weighted averages are nonexpansive as well.
Since the  expressions depend only on the products $\lambda_i\alpha_i$ we will introduce
the set
$$\mathcal B=\left \{\beta=(\beta_1,\ldots,\beta_N)\,\ \left|\, \beta_i\ge 0,\ \sum_{i=1}^N\beta_i\le 2\right.\right \}$$
and define for $\beta\in \mathcal B$
\begin{equation}\label{one}
Q_\beta=I-\sum_{j=1}^N \beta_j (I-P_{C_j}).
\end{equation}
\noindent{\bf Remark.}
Observe that for any $\beta\in \mathcal B$ there exist relaxation parameters $\alpha_1,\ldots,\alpha_N$ and
average parameters $\lambda_1,\ldots,\lambda_N$ such that $\beta_i=\lambda_i\alpha_i.$ Indeed, if
$\sum_{k=1}^N\beta_k>0$ we may set $\alpha_i= \sum_{k=1}^N\beta_k $ and
$\lambda_i=\beta_i\left(\sum_{k=1}^N\beta_k\right )^{-1}.$ On the other hand if $\beta_i=0$ for any $i$ we take
$\alpha_i=0$  and $\lambda_i=1/N $ for any $i.$ Therefore every operator $Q_\beta$ is   nonexpansive.
\noindent{\bf Remark.} All the results in this work remain valid if we replace $P_{C_i}$ with a firmly nonexpansive mappings
$T_i$ such that $T_i(c)=c$ for $c\in C_i$ (see \cite[p. 370]{BB}).
\section{Auxiliary results}
 \begin{pro}
For any $x\in \mathbb{R}^d$ and any $c\in C_1\cap C_2\cap\ldots \cap C_N$ we have
$$\|Q_{\beta}(x)-c\|^2\le \|x-c\|^2-\sum_{j=1}^N
 (2 -{\beta_j\over \kappa_j}  ) \beta_j \|x-P_{C_j}(x)\|^2,
$$ where $\kappa_1,\kappa_2,\ldots,\kappa_N$ are any nonnegative numbers such that
$\sum_{j=1}^N\kappa_j~=~1.$
\end{pro}
\noindent{\bf Remark.} We set $ \beta_j/\kappa_j=0$ whenever
$ \beta_j=0.$ If $\beta_j>0$ and $\kappa_j=0$ we set
$\beta_j/\kappa_j=+\infty.$
\begin{proof} With no loss of generality we may assume that $c=0.$ Then
by the convexity of each set $C_j$ we have
$$\langle P_{C_j}(x), x-P_{C_j}(x)\rangle\ge 0.$$ Hence
\begin{multline}
\|Q_{\beta}x \|^2=
\left \|x-\sum_{j=1}^N\beta_j(x-P_{C_j}(x))\right \|^2\\
=\|x \|^2+\left \|\sum_{j=1}^N\beta_j(x-P_{C_j}(x)) \right \|^2
 -2\sum_{j=1}^N\beta_j\langle x, x-P_{C_j}(x)\rangle\\
 =\|x \|^2+\left \|\sum_{j=1}^N\beta_j(x-P_{C_j}(x)) \right \|^2
 -2\sum_{j=1}^N\beta_j\,\|x-P_{C_j}(x)\|^2 \\
 -2\sum_{j=1}^N\beta_j\,\langle P_{C_j}(x), x-P_{C_j}(x)\rangle
\end{multline}
$$\le \|x \|^2+\left \|\sum_{j=1}^N\beta_j(x-P_{C_j}(x)) \right \|^2
 -2\sum_{j=1}^N\beta_j\,\|x-P_{C_j}(x)\|^2$$

Let $\kappa_1,\kappa_2,\ldots, \kappa_N$ satisfy the assumptions. Then by the convexity
of the function $x\mapsto\|x\|^2$ we obtain
\begin{multline}
\left \|\sum_{j=1}^N\beta_j(x-P_{C_j}(x)) \right \|^2=\left
\|\sum_{j=1}^N\kappa_j\kappa_j^{-1}\beta_j(x-P_{C_j}(x)) \right \|^2\\
\le \sum_{j=1}^N {\beta_j^2\over \kappa_j }\|x-P_{C_j}(x)\|^2.
\end{multline}
Combining (1) and (2) concludes the proof.
\end{proof}
\noindent{\bf Remark.} Setting
$$\kappa_j={\beta_j  \over
\sum_{k=1}^N\beta_j}
$$ implies
\begin{equation}
\|Q_{\beta}(x)-c\|^2\le \|x-c\|^2-  \left(2- {\textstyle\sum_{k=1}^N{\beta_k }}
  \right ) \sum_{j=1}^N \beta_j\|x-P_{C_j}(x)\|^2
\end{equation}
the inequality obtained in \cite[Lemma 3.2(ii)]{BB}.
 \begin{thm}Given   a family of convex
sets $C_1,\ldots, C_N$ with nonempty intersection $C.$ Let $\beta\in \mathcal B $
  and   $I\subset \{1,2,\ldots, N\}.$  Then for any $x\in \mathbb{R}^d$
 and any $c\in C$ we have
$$  \|Q_{\beta}(x)-c\|^2\le \|x-c\|^2- \min_{i\in I}{\nu_i}\,\max_{i\in I}
  d^2(x,C_i),$$
where
$$\nu_i={2\beta_i \left (2-\sum_{k=1}^N\beta_k\right )\over
 \beta_i +2-\sum_{k=1}^N \beta_k}.$$
 In particular the inequality holds when $I$ is the set of active indices, i.e.
 \begin{equation}
 I = \{ i\ |\ 1\le i\le N,\ \beta_i>0\}
 \end{equation}
\end{thm}
\begin{proof}
 Fix $i\in I$
  and set
$$
\kappa_j = \begin{cases}
{1\over 2}\beta_j  & j\neq i, \\
1-{1\over 2}\sum_{j\neq i} \beta_j & j=i.
\end{cases}
$$
On substituting theses values into the inequality of Proposition 1 we obtain
$$ \|Q_{\beta}(x)-c\|^2 \le \|x-c\|^2 -
 \nu_i\,d^2(x,C_i)\le \|x-c\|^2 -\min_{k\in I}\nu_k\,d^2(x,C_i). $$
 Now maximizing with respect to $i\in I$ gives the conclusion.
\end{proof}
\section{Main result}
\begin{thm}Fix  a family of convex
sets $C_1,\ldots, C_N$ with nonempty intersection $C.$ Given a sequence $ \beta^{(n)}\in
\mathcal B .$      Let  $I^{(n)}$ denote the set of active
indices for $\beta^{(n)}.$ Assume that every index $i\in
\{1,2,\ldots, N\}$ occurs in $I^{(n)}$ for infinitely many $n.$ Let $n_k$ be positive
integers such that $n_{k-1}<n_{k}$ and
$$\{1,2,\ldots,N\}\subset I^{(n_{k-1})}\cup I^{(n_{k-1}+1)}\cup \ldots \cup I^{(n_k-1)},$$
i.e. every index occurs at least once for $n$ such that $n_{k-1}\le n<n_k.$ For
$$\nu_i^{(n)}={2\beta_i^{(n)} \left (2-\sum_{k=1}^N\beta_k^{(n)}\right )\over
 \beta_i^{(n)} +2-\sum_{k=1}^N \beta_k^{(n)}}.$$
let
$$\nu^{(k)}=\min\{\nu_i^{(n)}\,|\,n_{k-1}<n \le n_k, i\in I^{(n)}\}.$$
Assume that
$$\sum_{k=1}^\infty \nu^{(k)} =+\infty.$$
Then for any $x^{(0)}\in \mathbb{R}^d$ the sequence $x^{(n)}$ defined as
$$ x^{(n)}=
Q_{\beta^{(n)}}(x^{(n-1)}),\qquad n\ge 1
 $$
is convergent  to a point in $C.$
\end{thm}
\noindent{\bf Remark}. This result generalizes \cite[Thm 3.20(ii)]{BB} (see also
\cite[Cor. 3.25]{BB}) in two essential aspects.
First of all it allows repetitve control while \cite[Thm 3.20]{BB} could afford only intermittent
control, i.e. when the sequence $n_k$ is of the form $n_k=kp.$ Secondly the coefficients
$\nu_i^{(n)}$ are smaller than $\mu_i^{(n)},$ introduces in \cite{BB}   as
$$ \mu_i^{(n)}= 2\beta_i^{(n)}\left (2-\sum_{k=1}^N\beta_k^{(n)} \right ).$$

 For example let $N=2$ and
$$\beta_1^{(n)}={1\over { n}},\ \beta_2^{(n)}=2-{2\over  { n}} .$$ The algorithm is
then 1-intermittent, hence
 we can take $n_k=k$ and $I_k=\{1,2\}.$ Thus
$$\sum_{k=1}^\infty \min\{\mu_1^{(k)},\mu_2^{(k)}\}<+\infty,\quad
\sum_{k=1}^\infty \min\{\nu_1^{(k)},\nu_2^{(k)}\}=+\infty.$$ Therefore Theorem 3.20 of
\cite{BB}  does not apply while our Theorem 2 does.
\begin{proof} With no loss of generality we may assume that $0\in C.$
Fix $u^{(0)}\in\mathbb{R}^d$  and let $u^{(n)}=Q_{\beta^{(n)}}(u^{(n-1)})$ for $n\ge 1.$ By Theorem
1 we get
\begin{equation}\label{5}
\|u^{(n)}\|^2\le \|u^{(n-1)}\|^2-\min_{i\in I^{(n)}}\nu_i^{(n)}\,\max_{i\in
I^{(n)}}d^2(u^{(n-1)},C_i)
\end{equation}
Iterating (\ref{5}) leads to
\begin{multline}\label{7}
\|u^{(n)}\|^2\le \|u^{(0)}\|^2-\sum_{m=1}^n\min_{i\in I^{(m)}}\nu_i^{(m)}\,\max_{i\in I^{(m)}}d^2(u^{(m-1)},C_i)\\
\le \|u^{(0)}\|^2- \min_{\stackrel{0< m\le n}{i\in I^{(m)}}}\nu_i^{(m)}\,
\max_{\stackrel{1\le m\le n}{i\in I^{(m)}}} d^2(u^{(m-1)},C_i)
\end{multline}
\begin{lem} Given   a family of convex
sets $C_1,\ldots, C_N$ with nonempty intersection $C$ and      a sequence  $
\beta^{(n)}\in \mathcal B.$   Let $I^{(n)}$ denote the set of
active indices for $\beta^{(n)}.$ Assume  that every index $i\in
\{1,2,\ldots, N\}$ occurs in $I^{(n)}$ for at least one $n.$ Then for any positive number
$R$ there exists a nondecreasing and positive function $\eta_R: (0,+\infty)\to
(0,+\infty)$
 such that
$$ \max_{\stackrel{n\ge 1}{i\in I^{(n)}}}{d^2(u^{(n-1)},C_i)}\ge \eta_R(d(u^{(0)},C))$$
for any $u^{(0)}$ with $\|u^{(0)}\|\le R$ and $u^{(0)}\notin C.$ The function $\eta_R$ is
independent of the choice of the sequence $ \beta^{(n)}.$
\end{lem}
\begin{proof} Fix $r>0$ and  consider the set
$$B_{r,R}=\{u^{(0)}\in \mathbb{R}^d\,:\, \|u^{(0)}\|\le R,\ d(u^{(0)},C)\ge r\}.$$
The proof will be completed if we show that for any $u^{(0)}\in B_{r,R}$ there exists a
positive number $\eta_R(r)$ such that
$$ \max_{\stackrel{n\ge 1}{i\in I^{(n)}}}{d^2(u^{(n-1)},C_i)}\ge \eta_R(r).$$
 Suppose, by
contradiction, that for any $m\in \mathbb N$ there exist  vectors $u^{(0)}_{(m)}\in B_R,$
 and a sequence
   $\beta^{(n)}_{(m)}\in \mathcal B,$
 satisfying the assumptions of Lemma 1, such that
\begin{equation}\label{crucial}
 \max_{\stackrel{n\ge 1}{i\in I^{(n)}_{(m)}}}{d^2(u_{(m)}^{(n-1)},C_i)}\le {1\over m}
\end{equation}
where $$u_{(m)}^{(n)}= Q_{
\beta_{(m)}^{(n)}}(u_{(m)}^{(n-1)}),\qquad n\ge 1.
$$
By compactness of $B_{r,R}$ we may assume that $x^{(m)}_0\stackrel{m}{\to} y$ and $y\in
B_{r,R}.$ Consider the sets $I_{(m)}^{(1)} .$ Some indices  of $\{1,2,\ldots, N\}$ occur
for infinitely many $m.$ Let $\mathcal A_1$ denote those indices. Clearly we may assume,
eventually by restricting to large values of $m,$ that only indices of $\mathcal A_1$ may
occur in $I_{(m)}^{(1)} ,$ and each index does so infinitely many $m.$ Therefore, fixing
$n=1$ and taking the limit in (\ref{crucial}) when $m\to\infty$  yield $y\in C_i$ for any
$i\in \mathcal A_1.$  If $\mathcal A_1=\{1,2,\ldots, N\},$ then $y\in C,$ which is a
contradiction with $y\in B_{r,R}.$ Otherwise we have $\mathcal A_1\subsetneq\{1,2,\ldots,
N\}.$ Since for any fixed $m$ every index of $ \{1,2,\ldots, N\}$ occurs in
$I_{(m)}^{(n)}$ at least for one $n,$ there exists the least number $l_m$ such that
$I_{(m)}^{(l_m)}\setminus \mathcal A_1\neq \emptyset.$ Observe that
$Q_{\beta}(y)=y$ for any   $\beta\in \mathcal B$ such that
the set of active indices $I$ of $\beta$ is contained in $\mathcal A_1.$
Therefore
\begin{equation}
y=Q_{\beta_{(m)}^{(l_m-1)}}Q_{\beta_{(m)}^{(l_m-2)}}\ldots Q_{\beta_{(m)}^{(2)}}Q_{\beta_{(m)}^{(1)}}(y).
\end{equation}
Define
\begin{equation}
\widetilde{u}_{(m)}^{(1)}:=u_{(m)}^{(l_m-1)}=Q_{\beta_{(m)}^{(l_m-1)}}Q_{\beta_{(m)}^{(l_m-2)}}\ldots Q_{\beta_{(m)}^{(2)}}Q_{\beta_{(m)}^{(1)}}(u_{(m)}^{(0)}).
\end{equation}
 Since the operators $Q_{\beta_{(m)}^{(k)}}$ are nonexpansive we obtain
\begin{equation}
 d(\widetilde{u}_{(m)}^{(1)},y)
\le d(u^{(0)}_{(m)},y)
\end{equation}
Hence $\widetilde{u}_{(m)}^{(1)}\stackrel{m}{\to}y.$

Consider now the sets $I_{(m)}^{(l_m)}.$ Each of these sets contains elements which do
not belong to $\mathcal A_1.$ Let $\mathcal A_2$ denote those indices outside $\mathcal
A_1$ which occur for infinitely many values of $m.$. By restricting to large values of
$m$ we may assume that only indices of $\mathcal A_1$ and $\mathcal A_2$ may occur in
$I_{(m)}^{(l_m)}$ and all indices of $\mathcal A_2$ occur for infinitely many values of
$m.$ Set $n=l(m)+1$ in (\ref{crucial}). Then since
$\widetilde{u}_{(m)}^{(1)}=u_{(m)}^{(l_m-1)}$ we get
\begin{equation}
  \max_{i\in I_{(m)}^{(l_m)}}d^2(\widetilde{u}_{(m)}^{(1)}, C_i)\le {1\over m}.
\end{equation}
Therefore for any $i\in \mathcal A_2$ we have
$$
   d^2(\widetilde{u}_{(m)}^{(1)}, C_i)\le {1\over m}
$$ for infinitely many $m.$
 As $\widetilde{u}_{(m)}^{(1)}$ tends to $y,$ when $m\to\infty,$ we obtain that $y\in C_i$ for any $i\in
\mathcal A_2.$ Therefore $y\in C_i$ for $i\in \mathcal A_1\cup \mathcal A_2.$

By repeating this argument at most $N$ times we get that $y\in C_i$ for any
$i=1,2,\ldots, N.$ Hence $y\in C$ which contradicts the fact that $y\in B_{r,R}.$
 \end{proof}
 Let's return to the proof of Theorem 2. With no loss of generality we may assume
 that $0\in C.$ Fix $R>0$ and assume that
 $\|x^{(0)}\|\le R.$ Then since every operator $Q_{\beta}$ is nonexpansive we obtain
 $\|x^{(n)}\|\le R$ for any $n.$ Assume that
 $$\{1,2,\ldots,N\}\subset \bigcup_{j=n_{k-1} }^{n_{k}-1 }I^{(j)}.$$ Then combining Lemma 1 with
 $u^{(0)}=x^{(n_{k-1})}$ and   formula
(\ref{7}) yields
 \begin{multline*}
\|x^{(n_k)}\|^2\le \|x^{(n_{k-1})}\|^2- \min_{\stackrel{n_{k-1}< n\le n_k}{i\in
I^{(n)}}}\nu_i^{(n)}\,\max_{\stackrel{n_{k-1}< n\le n_k}{i\in I^{(n)}}}
d^2(x^{(n-1)},C_i)\\
\le \|x^{(n_{k-1})}\|^2- \nu^{(k)}\,\,\eta_R(d(x^{(n_{k-1})},C))
 \end{multline*}
 This implies that the series
 $$ \sum_{k=1}^\infty \nu^{(k)} \,\,\eta_R(d(x^{(n_{k-1})},C))\,$$
 is convergent.
 Since the operators $Q_{\beta}$ are nonexpansive,
  the sequence $d(x^{(n)},C)$ is nonincreasing. Therefore, by assumptions made on
 the coefficients $\nu_i^{(n)},$ we obtain that $\eta_R(d(x^{(n)},C))\stackrel{n}{\to}~0.$
Hence $d(x^{(n)},C)\stackrel{n}{\to}~0.$ Since $x^{(n)}$ is bounded, it contains a
convergent subsequence $x^{(n_m)}.$ Denote its limit by $c.$ Then $c\in C.$  By
Proposition 1 the sequence $\|x^{(n)}-c\|$ is nonincreasing. Therefore, it tends to zero,
i.e. $x^{(n)}\stackrel{n}{\to}c.$
\end{proof}
\section{Perturbation}
\begin{pro}Given   a family of convex
sets $C_1,\ldots, C_N$ with nonempty intersection $C$ and      a sequence  $
\widetilde{\beta}^{(n)}\in \mathcal B.$   Assume that for any $m\in \mathbb{N}$ and $x\in \mathbb{R}^d $
the sequence
$$Q_{\widetilde{\beta}^{(n)}}Q_{\widetilde{\beta}^{(n-1)}}\ldots Q_{\widetilde{\beta}^{(m)}}(x) $$
is convergent to an element of $C $ as $n\to \infty.$ Let a sequence $ \beta^{(n)}\in \mathcal B$ satisfy
$$ \sum_{n=1}^\infty \sum_{i=1}^N |\beta_i^{(n)}-\widetilde{\beta}_i^{(n)}|<\infty.$$ Then for any
$x\in \mathbb{R}^d$ the sequence
$$ Q_{\beta^{(n)}}Q_{\beta
^{(n-1)}}\ldots Q_{\beta^{(1)}}(x)$$
is convergent to an element of $C $ as $n\to \infty.$
\end{pro}
\begin{proof}
By (\ref{one}) we have
\begin{multline}
d(Q_{\beta^{(k)}}(y),Q_{ {\widetilde{\beta}}^{(k)}}(y))
  \le \sum_{j=1}^N|\beta_j^{(k)}
-\widetilde{\beta}_j^{(k)}|\,\|P_{C_j}y-y\|\\
\le \sum_{j=1}^N|\beta_j^{(k)}
-\widetilde{\beta}_j^{(k)}|\, d(y,C).
\end{multline}
Denote for simplicity
$$Q_n:= Q_{\beta^{(n)}},\quad
\widetilde{Q}_n:=Q_{\widetilde{\beta}^{(n)}} $$
and $$ {x}^{(m)}=Q_mQ_{m-1}\ldots Q_{1}(x).$$
Then
\begin{multline*}
d( {x}^{(n)} ,\,C)
=d(Q_{n}Q_{n-1}\ldots Q_{m+1}({x}^{(m)}),\,C)\\
   \le \sum_{k=m+1}^n
 d(\widetilde{Q}_{n}\ldots  \widetilde{Q}_{k+1}\widetilde{Q}_{k} ( {x}^{(k-1)}),
\widetilde{Q}_{n}    \ldots \widetilde{Q}_{k+1}Q_{k}( {x}^{(k-1)}))\\
+ d(\widetilde{Q}_{n} \widetilde{Q}_{n-1}\ldots  \widetilde{Q}_{m+1}( {x}^{(m)}) ,\,C)\\
\le
\sum_{k=m+1}^n
 d( \widetilde{Q}_{k} ( {x}^{(k-1)}),
 Q_{k}( {x}^{(k-1)}))
\,+ \,d(\widetilde{Q}_{n} \widetilde{Q}_{n-1}\ldots  \widetilde{Q}_{m+1}( {x}^{(m)}) ,C)
\\
\le
\sum_{k=m+1}^n\sum_{j=1}^N|\beta_j^{(k)}
-\widetilde{\beta}_j^{(k)}|\, d(\widetilde{x}^{(k-1)}, C)\,+\,
d(Q_{n} {Q}_{n-1}\ldots  {Q}_{m+1}(\widetilde{x}^{(m)}) ,C)\\
\le \sum_{k=m+1}^n\sum_{j=1}^N|\beta_j^{(k)}
-\widetilde{\beta}_j^{(k)}|\, d(x, C)\,+\,
d(\widetilde{Q}_{n} \widetilde{Q}_{n-1}\ldots \widetilde{Q}_{m+1}( {x}^{(m)}) ,C)
\end{multline*}
Now the conclusion follows from the assumptions. Indeed, we may assume that $d(x,C)>0$ as
otherwise $x^{(n)}=x\in C$ for any $n.$ Let $m$ be large so that
$$ \sum_{k=m+1}^\infty \sum_{j=1}^N| {\beta}_j^{(k)}
-\widetilde{\beta}_j^{(k)}|< {\varepsilon\over 2d(x,C)}.$$ Next let $n$ be large so that
$$d(\widetilde{Q}_{n} \widetilde{Q}_{n-1}\ldots  \widetilde{Q}_{m+1}( {x}^{(m)}) ,C)<\varepsilon/2. $$
Thus $d(x^{(n)},C)<\varepsilon$ for $n$ large. Hence $d(x^{(n)},C)\to 0$ as $n\to \infty.$
This implies that $x^{(n)}$ tends to a point in $C$
(see the end of the proof of Theorem 2).
\end{proof}
\begin{cor}Fix  a family of convex
sets $C_1,\ldots, C_N$ with nonempty intersection $C.$ Given a sequence $ \beta^{(n)}\in
\mathcal B .$      Let  $I^{(n)}$ denote the set of active
indices for $\beta^{(n)}$ and let $J^{(n)}$ be a sequence of subsets of $I^{(n)}$
such that
$$ \sum_{n=1}^\infty \sum_{i\in I^{(n)}\setminus J^{(n)}}\beta^{(n)}_i<+\infty.$$
 Assume that every index $i\in
\{1,2,\ldots, N\}$ occurs in $J^{(n)}$ for infinitely many $n.$ Let $n_k$ be positive
integers such that $n_{k-1}<n_{k}$ and
$$\{1,2,\ldots,N\}\subset J^{(n_{k-1})}\cup J^{(n_{k-1}+1)}\cup \ldots \cup J^{(n_k-1)},$$
i.e. every index occurs at least once for $n$ such that $n_{k-1}\le n<n_k.$ For
$$\nu_i^{(n)}={2\beta_i^{(n)} \left (2-\sum_{k=1}^N\beta_k^{(n)}\right )\over
 \beta_i^{(n)} +2-\sum_{k=1}^N \beta_k^{(n)}} $$
let
$$\nu_J^{(k)}=\min\{\nu_i^{(n)}\,|\,n_{k-1} <n\le n_k, i\in J^{(n)}\}.$$
Assume
that
$$\sum_{k=1}^\infty\nu_J^{(k)} =+\infty.$$
Then for any $x^{(0)}\in \mathbb{R}^d$ the sequence $x^{(n)}$ defined as
$$ x^{(n)}=
Q_{\beta^{(n)}}(x^{(n-1)}),\qquad n\ge 1
 $$
is convergent  to a point in $C.$
\end{cor}
\begin{proof}
Define
$$ \widetilde{\beta}^{(n)}_{i}=\begin{cases}
\beta^{(n)}_{i} & \mbox{if}\ i\in J^{(n)},\\
0 & \mbox{if}\ i\in I^{(n)}\setminus J^{(n)}.
\end{cases}$$
Clearly we have
$$ \sum_{k=1}^N \widetilde{\beta}_k^{(n)}\le \sum_{k=1}^N \beta_k^{(n)}.$$
Hence $\widetilde{\nu}_i^{(n)}\ge \nu_i^{(n)}$ for $i\in J^{(n)},$ which implies
$\widetilde{\nu}^{(k)}\ge\nu_J^{(k)}. $
Therefore,   the sequence $\widetilde{\beta}^{(n)}$ satisfies the assumptions of Theorem 2. Consequently
the sequence
$$ Q_{\widetilde{\beta}^{(n)}}Q_{\widetilde{\beta}
^{(n-1)}}\ldots Q_{\widetilde{\beta}^{(m)}}(x)$$ is convergent to an element of $C $ as $n\to \infty.$
The sequences $\beta^{(n)}$ and $\widetilde{\beta}^{(n)}$ satisfy
$$ \sum_{n=1}^\infty \sum_{i=1}^N |\widetilde{\beta}_i^{(n)}-\beta_i^{(n)}|<\infty.$$
Thus applying Proposition 1 concludes the proof.
\end{proof}
\noindent{\bf Example.}
Consider $N=3$  and
$$\begin{array}{lcllcllcl}
\beta_1^{(2n)}&=&{1\over n^2},&\beta_2^{(2n)}&=& {1\over n},&\beta_3^{(2n)}&=&2-{2\over n},\bigskip\\
\beta_1^{(2n+1)}&=&{1\over n },& \beta_2^{(2n+1)}&=&{1\over n^2},& \beta_3^{(2n+1)}&=&2-{2\over n }.
\end{array}$$
The scheme is 1-intermittent, i.e. $I^{(n)}=\{1,2,3\}$ for any $n.$ Observe that the
 assumptions of Theorem 2 are not satisfied. Now,
consider this scheme as 2-intermittent and let
$$ J^{(2n)}=\{2,3\}\qquad J^{(2n+1)}=\{1,3\}.$$
Then we can apply Corollary 1 to obtain that this scheme leads to the convergence
of the algorithm.

\section{Intermittent control}
The assumptions of Theorem 2 depend on the behaviour of the coefficients $\beta_i^{(n)}$ where
$i\in I^{(n)},$ i.e. those coefficients which are positive. Roughly the conclusion holds
if these coefficients are not to small and the sums $s^{(n)}=\sum_{i=1}^N\beta_i^{(n)}$ do not
approach the value 2 too fast. By Corollary 1  we can allow
some small coefficients $\beta^{(n)}_i$ by using perturbation technique.
However in special case of intermittent control and when the sums $s^{(n)}$ stay away from 2
we can entirely liberate ourselves from assumptions on all positive coefficients $\beta_i^{(n)}.$
\begin{thm}Fix  a family of convex
sets $C_1,\ldots, C_N$ with nonempty intersection $C.$ Given a sequence $ \beta^{(n)}\in
\mathcal B  $
such that $$ s^{(n)}=\sum_{i=1}^N\beta_i^{(n)}\le 2-\varepsilon,$$
for some constant $\varepsilon>0.$      Let  $I^{(n)}$ denote the set of active
indices for $\beta^{(n)} $ and let $J^{(n)}$ be a sequence of subsets
of $I^{(n)}.$  Assume that  there is a positive integer $p$ such that for any $k$ we have
$$\{1,2,\ldots,N\}\subset J^{((k-1)p)}\cup J^{(k-1)p+1)}\cup \ldots \cup J^{(kp-1)}.$$
 Let
$$\nu^{(k)}_J=\min\{\beta_i^{(n)}\ |\ (k-1)p<n \le kp,\, i\in J^{(n)}\}.$$
Assume that
$$\sum_{k=1}^\infty \nu^{(k)}_J =+\infty.$$
Then for any $x^{(0)}\in \mathbb{R}^d$ the sequence $x^{(n)}$ defined as
$$ x^{(n)}=
Q_{\beta^{(n)}}(x^{(n-1)}),\qquad n\ge 1
 $$
is convergent  to a point in $C.$
\end{thm}
\begin{proof}
First observe that since $s^{(n)}\le 2-\varepsilon$ we have
$${2\varepsilon\over 2+\varepsilon}\,\beta_i^{(n)}\le\nu_i^{(n)}={2\beta^{(n)}_i(2-s^{(n)})\over \beta^{(n)}_i+2-s^{(n)}}
\le 2\beta^{(n)}_i.
$$
Therefore we can replace the coefficients $\nu^{(n)}_i$ with $\beta^{(n)}_i $ when applying Theorem 2 and Corollary 1.

Let $n_k=kp.$ If for
$$\nu^{(k)}=\min\{\beta_i^{(n)}\,|\,(k-1)p<n \le kp, i\in I^{(n)}\}$$
we have
\begin{equation}\label{infty}\sum_{k=1}^\infty \nu^{(k)}=+\infty
\end{equation}
we can apply Theorem 2 to get the conclusion. Thus it suffices to consider the case when
\begin{equation}\label{last}
\sum_{k=1}^\infty \nu^{(k)}<+\infty. \end{equation}
Let
$$A=\{k\in \mathbb{N}\ |\ (\exists\, n)( \exists\, i)\ (k-1)p<n\le kp,\ i\in I^{(n)}\setminus J^{(n)},\
\nu^{(k)}=\beta_i^{(n)}\}.$$
For every $k\in A$ choose $n_k$ and $i_k$ such that
$$ (k-1)p<n_k\le kp,\quad i_k\in I^{(n)}\setminus J^{(n)},\quad \nu^{(k)}=\beta_{i_k}^{(n_k)}.$$
By (\ref{last}) we have
$$ \sum_{k\in A}^\infty \beta_{i_k}^{(n_k)}<+\infty.$$
Define the new coefficients $\widetilde{\beta}_i^{(n)}$ by nullifying the coefficients $\beta_i^{(n)}$
for $i=i_k$ and $n=n_k,$ i.e. let
$$
\widetilde{\beta}_i^{(n)}=\begin{cases}0 & \mbox{if}\ n=n_k,\ i=i_k\ \mbox{for some}\ k\\
\beta_i^{(n)} & \mbox{otherwise}
\end{cases}
$$
By construction the sums $\widetilde{s}^{(k)}$ stay away from 2 since $\widetilde{s}^{(k)}\le s^{(k)}.$
Moreover $J^{(n)}\subset \widetilde{I}^{(n)},$ where
$\widetilde{I}^{(n)} $ denote the set of active indices for $\widetilde{\beta}^{(n)}.$ By Corollary 1  the convergence of the algorithm   for the new
coefficients implies its convergence for the original ones. Thus we can restrict ourselves to the coefficients
$\widetilde{\beta}^{(n)}_i.$ Clearly for $i\in J^{(n)}$ we have
$\widetilde{\beta}^{(n)}_i= {\beta}^{(n)}_i.$
If the new coefficients satisfy (\ref{infty}) we are done by Theorem 2. If not, we can perform the same transformation
as before.  After at most $pN$ iterations we will obtain a sequence to which we can apply Theorem 2 and which differs from
original sequence as in Corollary~1.
\end{proof}

\noindent{\bf Example}
Let $N=3$ and
$$\begin{array}{lcllcllcl}
\beta_1^{(3n)}&=&1, &\beta_2^{(3n)}&=& {1\over n },& \beta_2^{(3n)}&=& {1\over n^2},\bigskip\\
\beta_1^{(3n+1)}&=& {1\over n},& \beta_2^{(3n+1)}&=&1,&\beta_2^{(3n+1)}&=& {1\over n^2},\bigskip\\
\beta_1^{(3n+2)}&=& {1\over n^2},& \beta_2^{(3n+2)}&=&{1\over n},&\beta_2^{(3n+2)}&=& 1.
\end{array}$$
We have $I^{(n)}=\{1,2,3\}.$ Let
$$ J^{(3n)}=\{1\},\quad J^{(3n+1)}=\{2\},\quad J^{(3n+2)}=\{1\}$$
and $p=3.$   We have $\nu_J^{(k)}=1$ for any $k.$ Hence all the assumptions of Theorem 3 are satisfied.

\noindent{\bf Acknowledgment.} I am very grateful to Andrzej Cegielski for his help, and in particular
  for turning my attention to the paper \cite{BB}.

\end{document}